\documentclass[12pt]{amsart}
\usepackage{amsmath}
\usepackage{amssymb}
\usepackage{enumerate}
\usepackage{amsbsy}
\usepackage{amsfonts}
\usepackage{amsthm}

\newtheorem{theorem}{Theorem}[section]
\newtheorem{proposition}[theorem]{Proposition}
\newtheorem{lemma}[theorem]{Lemma}

\newtheorem{corollary}[theorem]{Corollary}

\numberwithin{equation}{section}

\newtheorem{remark}{Remark}

\begin{document}

\title[Orbital stability of FNLS]{On the orbital stability of fractional Schr\"{o}dinger equations}

\author[Y. Cho]{Yonggeun Cho}
\address{Department of Mathematics, and Institute of Pure and Applied Mathematics, Chonbuk National University, Jeonju 561-756, Republic of Korea}
\email{changocho@jbnu.ac.kr}
\author[G. Hwang]{Gyeongha Hwang}
\address{Department of Mathematical Sciences, Seoul National University, Seoul 151-747, Republic of Korea}
\email{ghhwang@snu.ac.kr}
\author[H. Hajaiej]{Hichem Hajaiej} \address{Department of Mathematics, King
Saud University, P.O. Box 2455, 11451 Riyadh, Saudi Arabia}\email{hhajaiej@ksu.edu.sa}
\author[T. Ozawa]{Tohru Ozawa}\address{Department of Applied Physics, Waseda University, Tokyo 169-8555, Japan}\email{txozawa@waseda.jp}

\maketitle


\begin{abstract}
We show the existence of ground state and  orbital stability of standing waves of fractional Schr\"{o}dinger equations with power type nonlinearity. For this purpose we establish the uniqueness of weak solutions.
\end{abstract}





\newcommand{\p}{\Phi}

\section{Introduction}



 In this paper we consider the following Cauchy problem:
\begin{align}\label{main eqn}\left\{\begin{array}{l}
i\partial_t\p + (-\Delta)^s \p = N(x, \p)\;\;\mbox{in}\;\;\mathbb{R}^{1+n},\\
\p(0, x) = \varphi(x)\;\; \mbox{in}\;\;\mathbb{R}^n.
\end{array}\right.
\end{align}
Here $n \ge 1$, $0 < s < 1$, $\p: \mathbb R^{1+n} \to \mathbb C$ and $N: \mathbb R^n \times \mathbb C \to \mathbb C$.

The equation \eqref{main eqn}, called {\it fractional} nonlinear Schr\"{o}dinger equation, appears in many fields in science and engineering. Other  domains of  applications of such equations, involving the  fractional powers of  the  Laplacian, arise in medicine (RMI and heart  diseases). It is also of a great importance in astrophysics, signal processing, turbulence, and water waves, where the cases $s=\frac14$ and $s = \frac34$ are the most relevant (see \cite{hmow} and  references therein).

We will focus our attention on the orbital stability of standing waves of this Schr\"{o}dinger equation. Our results generalize those of \cite{bori} and \cite{guhu}. The paper \cite{bori} seems to be the first one dealing with the orbital stability in the fractional case. The authors studied \eqref{main eqn} for $s = \frac12, n = 1$ with an autonomous cubic power nonlinearity. In \cite{guhu} the authors extended the previous paper to general nonlinearities for $n \ge 2$, but without showing the uniqueness of weak solutions. 
In our work, we use the concentration-compactness lemma to prove the orbital stability of standing waves, as stated by Cazenave and Lions, but without introducing a problem at infinity. See Proposition \ref{orb1} below. We also establish the uniqueness of weak solutions to the Cauchy problem \eqref{main eqn} under suitable conditions on $V$ and $f$. Unlike the usual Schr\"odinger equation ($s = 1$), it is not an easy matter to show the uniqueness in the fractional setting, since we cannot utilize the standard Strichartz estimates due to a regularity loss (\cite{cox}). Here we exploit weighted Strichartz estimates without regularity loss. Instead, some integrability conditions on $a, V$ are necessary to treat the weights. The details of uniqueness of weak solutions will be discussed in Section \ref{wp}.


To present our results let us set $N(x, \p) = V(x)\p + f(x, \p)$ and describe assumptions: The functions $V: \mathbb R^n \to \mathbb R$ and $f : \mathbb R^n \times \mathbb C \to \mathbb C$ are measurable and $f$ satisfies that
$f(x, z) = \frac{z}{|z|}f(x, |z|)$ for $z \in \mathbb C\setminus \{0\}$, and for some $\ell$ with $0 < \ell \le \frac{4s}{n-2s}$ and nonnegative measurable functions $a, b$
\begin{align}
|f(x, z)| &\le a(x)|z|^{\ell+1},\label{cond-f}
\end{align}
for all $x \in \mathbb R^n$
and \begin{align}
|f(x, z_1) - f(x, z_2)| &\le b(x)(|z_1|^\ell + |z_2|^\ell )|z_1 - z_2|.\label{cond-f2}
\end{align}
for all $x \in \mathbb R^n$ and $z_1, z_2 \in \mathbb C$.
 Let us also set $F(x, |\p|) = \int_0^{|\p|} f(x, \alpha)\,d\alpha$. Then we define a functional $J$ by
$$
J(\p) = \frac12 \|(-\Delta)^\frac s2 \p\|_{L^2}^2 - \frac12 \int V(x)|\p|^2\,dx- \int F(x, |\p|)\,dx
$$
and also $M$ by $M(\p) = \int |\p|^2\,dx$.
By a standing wave of \eqref{main eqn} we mean a solution $\p(t, x)$ of the form $e^{i\omega t} u$ for some $\omega \in \mathbb R$, where $u$ is a solution of the equation
\begin{align}\label{ell}(-\Delta)^s u - \omega u = V(x)u + f(x, u). \end{align}
Some authors have studied the existence of $u$ under suitable conditions on $f$.
For this purpose they showed that if $(u_k)$ is a minimizing sequence of the problem
$$
I_\mu = \inf \{J(u) : u \in S_\mu\}, \quad S_\mu = \{u \in H^s(\mathbb R^n, \mathbb C): M(u) = \mu\}
$$
with a prescribed positive number $\mu$, then $u_k \to u$ in $H^s$ up to a subsequence, where $u$ is a solution of \eqref{ell} for some $\omega$.
Now by following the definition of Cazenave-Lions, we set
$$
\mathcal O_\mu = \{ u \in S_\mu : J(u) = I_\mu\}.
$$

Our first result is the existence of ground states.
\begin{proposition}\label{orb1}
Let $n \ge 1$, $0 < s < 1$ and $0 < \ell < \frac{4s}{n}$. Suppose that $$0 \le V \in L_{loc}^{p_1} + L^{p_2}(|x| > 1)$$ for $\frac{n}{2s}< p_1, p_2 < \infty$ and $f$ satisfies \eqref{cond-f} with $a \in L_{loc}^{q_1} + L^{q_2}(|x| > 1)$ for $\frac{2n}{4s-n\ell} < q_1, q_2 < \infty$, and that there exist $\kappa, R, N, \delta > 0$ and $\beta, \sigma > 0$ such that
$$\frac{n\beta}2 + \delta - 2s < 0$$ and
\begin{align}\label{cond-F1}
F(x, |z|) \ge \kappa |x|^{-\delta}|z|^{2+\beta},
\end{align}
for any $|z| \le N$ and $|x| \ge R$ and \begin{align}\label{cond-F2}F(x, \theta |z|) \ge \theta^{2+\sigma} F(x, |z|)\end{align} for all $x, z, \theta > 1$. Then $\mathcal O_\mu$ is not empty for any $\mu > 0$.
When $\ell = \frac{4s}{n}$, we assume that $a \in L^q(|x| \ge 1) \cap L^\infty$ for some $q$ with $\frac{n^2}{4s^2}< q < \infty$ if $n \ge 2$ and $1 < q < \infty$ if $n = 1$. Then $\mathcal O_\mu$ is not empty  for sufficiently small $\mu > 0$.
\end{proposition}

When $a \in L^\infty$, by adding some natural asymptotic conditions on $V$ and $f$, one can show the existence of ground state. See \cite{fqt} and \cite{haj}. But if we assume the radial symmetry of $V(x), f(x, \cdot)$ in $x$, then using the compactness of embedding $H_{rad}^s \hookrightarrow L^p$, $\frac12 < s < \frac n2$, $2 < p < \frac{2n}{n-2s}$(see \cite{chooz-ccm}), we have the following.
\begin{proposition}\label{orb2}
Let $n \ge 2$, $\frac12 < s < 1$ and $0 < \ell < \frac{4s}{n}$. Let $0 \le V \in L_{rad, loc}^\infty$ satisfy
\begin{align}\label{cond-V}\int_{|x| > 1}V(x)|x|^{-(n-2s)}\,dx < \infty\end{align} and $f(x, \cdot)$ be radially symmetric with $a \in L^\infty$. If $F$ satisfies \eqref{cond-F1} and \eqref{cond-F2}, then $\mathcal O_\mu \cap H_{rad}^s$ is not empty for any $\mu > 0$. When $\ell = \frac{4s}{n}$,  $\mathcal O_\mu$ is not empty  for sufficiently small $\mu > 0$.
\end{proposition}

We say that $\mathcal O_\mu$ is stable if it is not empty and satisfies that
for any $\varepsilon > 0$, there exists a $\delta > 0$ such that if $\varphi \in H^s$ with
$$
\inf_{u \in \mathcal O_\mu} \|\varphi - u\|_{H^s} < \delta,
$$
then
$$
\inf_{u \in \mathcal O_\mu}\|\p(t, \cdot) - u\|_{H^s} < \varepsilon
$$
for all $t \in [-T_1, T_2]$. Here $\p$ is the unique solution to \eqref{main eqn} in $C([-T_1, T_2]; H^s)$ with $M(\p(t)) = M(\varphi)$ and $J(\p(t)) = J(\varphi)$ for all $t \in [-T_1, T_2]$.

Let us introduce our main result.
\begin{theorem}\label{stability1}
Suppose that $s, \ell$, $V$ and $f$ satisfy all conditions as in Proposition \ref{orb1}, and that \eqref{main eqn} has the unique solution in $C([-T_1, T_2]; H^s)$ with $M(\p(t)) = M(\varphi)$ and $J(\p(t)) = J(\varphi)$ for all $t \in [-T_1, T_2]$. Then $\mathcal O_\mu$ is stable.
\end{theorem}

\begin{theorem}\label{stability2}
Suppose that $s, \ell$, $V$ and $f$ satisfy all conditions as in Proposition \ref{orb2}, and that \eqref{main eqn} has the unique solution in $C([-T_1, T_2]; H^s)$ with $M(\p(t)) = M(\varphi)$ and $J(\p(t)) = J(\varphi)$ for all $t \in [-T_1, T_2]$. Then $\mathcal O_\mu \cap H_{rad}^s$ is stable.
\end{theorem}

In view of the well-posedness results in Section \ref{wp} below, by assuming that $V, a, b$ are smooth and have suitable decay at infinity, we get the orbital stability for $\frac12 < s < 1$, $0 < \ell \le \frac{4s}n$ and $\ell < \ell_0$, where $\ell_0 = \infty$, if $n = 1$, $\frac{2s-1 }{2s(1-s)}$ if $n = 2$ and $\frac{n(2s-1)}{(n-2s)(n-1)}$ if $n \ge 3$. The critical case $\ell = \frac{4s}{n}$ can be included when $n = 1, 2, 3$.

Our paper is organized as follows. In Section \ref{ground} we will prove the existence of ground states by showing the compactness of the minimizing sequences of the constrained variational problem. This is  a key step to show the orbital stability of standing waves. This goal is achieved in Theorem \ref{stability1} and Theorem \ref{stability2}, which will be shown in Section \ref{orb}.  In the last section, we will discuss the uniqueness of solutions of the Cauchy problem for a large class of nonlinearities.

\section{Ground state}\label{ground}
\subsection{Proof of Proposition \ref{orb1}} If $0 < \ell < \frac{4s}n$, $p_1, p_2 > \frac n{2s}$ and $q_1, q_2 > \frac{2n}{4s-n\ell}$, then from  Gagliardo-Nirenberg's and Young's inequalities it follows that for any $u \in S_\mu$ there exist $\lambda$ such that
\begin{align}\begin{aligned}\label{lowbound}
J(u) &= \frac12\|(-\Delta)^\frac s2 u\|_{L^2}^2 - \int V(x)|u|^2\,dx - \int F(x, |u|)\,dx\\
&\ge \frac12\|(-\Delta)^\frac s2 u\|_{L^2}^2 - \|V\|_{L^{p_1}(|x| \le 1)}\|u\|_{L^{2p_1'}}^2 - \|V\|_{L^{p_2}(|x| > 1)}\|u\|_{L^{2p_2'}}^2\\
 &\qquad- \|a\|_{L^{q_1}(|x| \le 1)}\|u\|_{L^{(\ell+2)q_1'}}^{\ell+2}- \|a\|_{L^{q_2}(|x| > 1)}\|u\|_{L^{(\ell+2)q_2'}}^{\ell+2}\\
 &\ge \frac14\|(-\Delta)^\frac s2 u\|_{L^2}^2 - \lambda \sum_{i}\mu^{\theta_i} \;(\theta_1 = \theta_2 = 2, \theta_3 = \theta_4 = \ell).
\end{aligned}\end{align}
Thus $I_\mu  > -\infty$ for all $\mu > 0$. If $0 < \ell < \frac{4s}n$, $p_1, p_2 > \frac n{2s}$, $a \in L^\infty$ then there exist $\lambda, \theta_i > 0$, $i = 1, 2$
\begin{align}\begin{aligned}\label{lowbound2}
J(u) &= \frac12\|(-\Delta)^\frac s2 u\|_{L^2}^2 - \int V(x)|u|^2\,dx - \int F(x, |u|)\,dx\\
&\ge \frac12\|(-\Delta)^\frac s2 u\|_{L^2}^2 - \|V\|_{L^{p_1}(|x| \le 1)}\|u\|_{L^{2p_1'}}^2 - \|V\|_{L^{p_2}(|x| > 1)}\|u\|_{L^{2p_2'}}^2\\
 &\qquad\qquad\qquad\;\;\quad- \|a\|_{L^\infty}\|u\|_{L^{(\ell+2)}}^{\ell+2}\\
 &\ge \frac14\|(-\Delta)^\frac s2 u\|_{L^2}^2 - \lambda \sum_{i=1,2}\mu^{\theta_i} - C\|a\|_{L^\infty}\mu^\frac{4s}{n}\|(-\Delta)^\frac s2 u\|_{L^2}^2.
\end{aligned}\end{align}
So, if $C\|a\|_{L^\infty}\mu^\frac{4s}{n} < \frac14$, then $I_\mu > -\infty$.

 We show that
\begin{align}\label{neg}
I_\mu < 0\;\;\mbox{ for all}\;\; \mu > 0.
\end{align}
In fact, for $0 < \lambda \ll1$ letting $\psi_\lambda(x) = \lambda^\frac n2 \psi(\lambda x)$ for a nonnegative, rapidly decreasing radial smooth function $\psi$ in $S_\mu$ with $\psi \le N$, we see that $\psi_\lambda \in S_\mu$ and
\begin{align*}
J(\psi_\lambda) &= \frac12\|(-\Delta)^\frac s2 \psi_\lambda\|_{L^2}^2 - \frac12\int V(x)(\psi_\lambda)^2\,dx - \int F(x, \psi_\lambda)\,dx\\
&\le \lambda^{2s}\frac12\|(-\Delta)^\frac s2 \psi\|_{L^2} - \kappa \lambda^{\frac{n(\beta+2)}{2} - n + \delta} \int_{|x| \ge R}|x|^{-\delta}(\psi(x))^\beta\,dx.
\end{align*}
Since $0 < \lambda \ll 1$ and $\psi$ is smooth and rapidly decreasing, there exist constants $C_1, C_2 > 0$ such that
$$
J(\psi_\lambda) \le \lambda^{2s}(C_1 - \lambda^{\frac{n\beta}2 + \delta - 2s}C_2),
$$
which is strictly negative from \eqref{cond-F1} if $\lambda$ is sufficiently small.

On the other hand, from the proof of Lemma 3.1 of \cite{haj} one can easily show that $I_\mu$ is continuous on $(0, \infty)$.

For each $\mu > 0$ and $\theta > 1$ we take $\varepsilon < -I_\mu(1 - \theta^{-\frac\sigma2})$ and $v \in S_\mu$ such that $I_\mu  < J(v) < I_\mu + \varepsilon$. Then from \eqref{cond-F2} it follows that $$I_{\theta \mu} \le J(\sqrt\theta v)\le \theta^{1+\frac\sigma2}J(v) \le \theta^{1+\frac\sigma2}(I_\mu + \varepsilon) < \theta I_\mu,$$ which means
\begin{align}\label{subadd}
I_{\mu} < I_\nu + I_{\mu-\nu}\;\;\mbox{for all}\;\; 0 < \nu < \mu.
\end{align}
For this see \cite{lion}.

Let $(u_j) \subset S_\mu$ be a minimizing sequence such that $J(u_j) \to I_\mu$. From \eqref{lowbound} we deduce that $(u_j)$ is bounded in $H^s$. To show $\mathcal O_\mu \neq \varnothing$ we will use the concentration-compactness (see \cite{lion}). Let the concentration function $\mathfrak m_j$ be defined by
$$
\mathfrak m_j(r) = \sup_{y \in \mathbb R^n} \int_{|x-y|<r} |u_j(x)|^2\,dx\;\;\mbox{for}\;\;r > 0.
$$
Set
$$
\nu = \lim_{r \to \infty}\liminf_{j \to \infty}\mathfrak m_j(r).
$$
Then $0\le \nu \le \mu$ and there exists a subsequence $u_j$ (still denoted by $u_j$) satisfying the following properties\footnote{One can verify the concentration-compactness by following the arguments in \cite{lion} or \cite{caz}. We omit the details.}.
\begin{enumerate}
\item If $\nu = 0$, then $\|u_j\|_{L^p} \to 0$ as $j \to \infty$ for all $p$ with $2 < p < s^*$, $s^* = \frac{2n}{n-2s}$ if $ n > 2s$ and $s^* = \infty$ if $n = 1$ and $\frac12 \le s < 1$.
\item If $\nu = \mu$, then there exists a sequence $(y_j) \subset \mathbb R^n$ and $u \in H^s$ such that for any $p$ with $2 \le p < s^*$
$$u_j(\cdot +y_j) \to u \;\;\mbox{as}\;\;j \to \infty \;\;\mbox{in}\;\; L^p$$
and given $\varepsilon > 0$ there exists $j_0(\varepsilon)$ and $r(\varepsilon)$ such that
$$\int_{|x-y_j| < r(\varepsilon)} |u_j|^2\,dx \ge \mu - \varepsilon, \;\;\mbox{whenever}\;\;j \ge j_0(\varepsilon).$$
\item If $0 < \nu < \mu$, then there exist $(v_j), (w_j) \subset H^s$ such that
\begin{eqnarray}
&{\rm supp}\;v_j \cap {\rm supp} \;w_j = \varnothing,\label{supp}\\
&\|v_j\|_{H^s} + \|w_j\|_{H^s} \le C \|u_j\|_{H^s},\label{norm}\\
& \lim_{j\to\infty}M(v_j) = \nu,\quad \lim_{j\to\infty}M(w_j) = \mu-\nu,\label{mass-sep}\\
&\liminf_{j\to \infty}\left( \|(-\Delta)^\frac s2 u_j\|_{L^2}^2 - \|(\Delta)^\frac s2 v_j\|_{L^2}^2- \|(\Delta)^\frac s2 w_j\|_{L^2}^2\right) \ge 0,\label{energy-sep}\\
&\lim_{j\to\infty}\|u_j - v_j - w_j\|_{L^p} = 0,\;\;2 \le p < s^*\label{potential-sep}.
\end{eqnarray}

\end{enumerate}

If $\nu = 0$, then for $0 < \ell < \frac{4s}{n}$ we have
\begin{align*}
&\frac12\int V(x)|u_j|^2\,dx + \int F(x, |u_j|)\,dx\\
& \le C\|V\|_{L^{p_1} + L^{p_2}}\|u_j\|_{L^{2p_1'} + L^{2p_2'}}^2 + C\|a\|_{L^{q_1} + L^{q_2}}\|u_j\|_{L^{(\ell+2)q_1'} + L^{(\ell+2)q_1'}} \to 0
\end{align*}
as $j \to \infty$. Here $L^{l_1} + L^{l_2}$ denotes $L^{l_1}(|x| \le 1) + L^{l_2}(|x| > 1)$.
For $\ell = \frac{4s}n$ we also have
\begin{align*}
&\frac12\int V(x)|u_j|^2\,dx + \int F(x, |u_j|)\,dx\\
& \le C\|V\|_{L^{p_1} + L^{p_2}}\|u_j\|_{L^{2p_1'} + L^{2p_2'}}^2 + C\|a\|_{L^\infty}\|u_j\|_{L^{\frac{2n+4s}n}} \to 0
\end{align*}
as $j \to \infty$. This implies $I_\mu = \lim_{j\to\infty}J(u_j) \ge \frac12\liminf\|(-\Delta)^\frac s2 u_j\|_{L^2}^2 \ge 0$ and contradicts \eqref{neg}.

If $0 < \nu < \mu$, then from the support condition \eqref{supp} it follows that
\begin{align*}
&J(u_j) - J(v_j) - J(w_j)\\
&= \frac12\left( \|(-\Delta)^\frac s2 u_j\|_{L^2}^2 - \|(\Delta)^\frac s2 v_j\|_{L^2}^2- \|(\Delta)^\frac s2 w_j\|_{L^2}^2\right)\\
&\quad - \frac12 \int V(x)(|u_j|^2 - |v_j + w_j|^2)\,dx  - \int \left(F(x, |u_j|) - F(x, |v_j + w_j|) \right)\,dx.
\end{align*}
From \eqref{energy-sep} and \eqref{potential-sep} we deduce that
$$
\liminf_{j\to \infty}(J(u_j) - J(v_j) - J(w_j)) \ge 0
$$
and thus
$$
I_\mu = \lim_{j\to\infty}J(u_j) \ge \liminf_{j\to\infty}J(v_j) + \liminf_{j\to\infty}J(w_j).
$$
Since $M(v_j) \to \nu$ and $M(w_j) \to \mu-\nu$, by the continuity of $I_\mu$ on $(0, \infty)$ we get
$$
I_\mu \ge I_\nu + I_{\mu-\nu},
$$
which contradicts \eqref{subadd}.

Therefore $\nu = \mu$. Set $\widetilde u_j(x) = u_j(x + y_j)$. Then $u, \widetilde u_j \in S_\mu$ and $\widetilde u_j \to u$ in $L^p$ for all $2 \le p < s^*$. On the other hand, $(u_j)$ is bounded in $H^s$. So, there is a subsequence (still denoted by $u_j$) converging to $v$ weakly in $H^s$ and strongly in $L_{loc}^p$ for any $1\le p < s^*$. Now for any $\varepsilon > 0$ we can find $R, j_0 > 1$ such that
\begin{align*}
&\int_{|x| > R}V(x)(|u_j|^2 + |v|^2)\,dx \le C\|V\|_{L^{p_2}(|x|>R)}< \frac\varepsilon4,\\
&\int_{|x| \le R}V(x)(|u_j| + |v|)|u_j -v|\,dx \le C\|V\|_{L^{p_1}(|x|\le R)}\|u_j - v\|_{L^{2p_1'}(|x| \le R)} < \frac\varepsilon4,
\end{align*}
(when $0 < \ell < \frac{4s}{n}$)
\begin{align*}
&\int_{|x| > R}a(x)(|u_j|^{\ell+2} + |v|^{\ell+2})\,dx \le C\|a\|_{L^{q_2}(|x| > R)} < \frac\varepsilon4,\\
&\int_{|x| \le R}a(x)(|u_j|^{\ell+1} + |v|^{\ell+1})|u_j - u|\,dx \le C\|a\|_{L^{q_1}(|x| \le R)}\|u_j - v\|_{L^{(\ell+2)q_1'}} < \frac\varepsilon4,
\end{align*}
(when $\ell = \frac{4s}{n}$)
\begin{align*}
&\int_{|x| > R}a(x)(|u_j|^{\ell+2} + |v|^{\ell+2})\,dx \le C\|a\|_{L^{q}(|x| > R)} < \frac\varepsilon4,\\
&\int_{|x| \le R}a(x)(|u_j|^{\ell+1} + |v|^{\ell+1})|u_j - u|\,dx \le C\|a\|_{L^{\infty}(|x| \le R)}\|u_j - v\|_{L^{\frac{2n+4s}{n}}} < \frac\varepsilon4,
\end{align*}
if $j > j_0$. Set $P(w) \equiv J(w) - \frac12\|(-\Delta)^\frac s2 w\|_{L^2}^2$. Then $P(u_j) \to P(v)$ as $j \to \infty$.
Suppose that $(y_j)$ is unbounded. Then up to subsequence we may assume that $|y_j|\to \infty$.
Since $\widetilde u_j \to u$ in $L^2$, $u_j - u(\cdot - y_j) \to 0$ in the sense of distributions. But $u(\cdot-y_j) \to 0$ and $u_j \to v$ in the sense of distributions and thus $v = 0$. That is, $P(u_j) \to 0$ as $j \to \infty$. This implies that $I_\mu = \lim_{j\to \infty}J(u_j) \ge 0$, which contradicts \eqref{neg}. So, $(y_j)$ is bounded. Now let $R^* = \sup_{j \ge 1} |y_j|$. Then for any $\varepsilon > 0$ we have
$$
\int_{|x| < R^* + r(\varepsilon)} |u_j|^2\,dx \ge \int_{|x-y_j| < r(\varepsilon)} |u_j|^2\,dx  \ge \mu - \varepsilon, \;\;\mbox{if}\;\; j \ge j_0.
$$
and thus
$$
M(v) \ge \int_{|x| < R^* + r(\varepsilon)} |v|^2\,dx \ge \lim_{j \to \infty}\int_{|x| < R^* + r(\varepsilon)} |u_j|^2\,dx \ge \mu-\varepsilon.
$$
This means $M(v) \ge \mu$, while the semi-continuity of weak limit implies $M(v) \le \mu$. Then $v \in S_\mu$. Since $P(u_j) \to P(v)$, we have \begin{align}\label{comparable}I_\mu \le J(v) \le \liminf \|(-\Delta)^\frac s2 u_j\|_{L^2}^2 + P(v) = \liminf(J(u_j)) = I_\mu.\end{align} Therefore $J(v) = I_\mu$. This completes the proof of Proposition \ref{orb1}.

\newcommand{\src}{\overline{S}_\mu}
\newcommand{\irc}{\overline{I}_\mu}
\subsection{Proof of Proposition \ref{orb2}}
Set $\src = \{v \in H_{rad}^s : M(v) = \mu\}$ and $\irc = \inf_{v \in \src}J(v)$.

From Proposition 1 of \cite{chooz-ccm} we see that
$$
|u(x)| \le C |x|^{-(\frac n2-s)}\|(-\Delta)^\frac s2 u\|_{L^2}
$$
a.e $x \in \mathbb R^n$, if $u \in H_{rad}^s$, $\frac12 < s < \frac n2$. So, given $\varepsilon > 0$ we can find an $R > 0$ such that
$\int_{|x| > R} V (x)|x|^{-(n-2s)}\,dx < \varepsilon$. From this we get
$$
\int_{|x| > R}V(x)|u|^2\,dx \le C\int_{|x| > R} V(x)|x|^{-(n-2s)}\,dx\|(-\Delta)^\frac s2 u\|_{L^2}^2 \le C\varepsilon \|(-\Delta)^\frac s2 u\|_{L^2}^2.
$$
In view of \eqref{lowbound} and \eqref{lowbound2}, $\irc > -\infty$. From \eqref{cond-F1} and \eqref{cond-F2} it follows that $\irc < 0$ and $I_\mu < I_\nu + I_{\mu-\nu}$ for all $0 < \nu < \mu$. Now let us take a minimizing sequence $(u_j)\subset \src$ such that $J(u_j) \to \irc$. Then from the concentration-compactness we get $(1), (2), (3)$ with $y_j = 0$. By the same argument as above it follows that there exists a function $u \in H^s$ such that $J(u) = \irc$, provided we can show that there exists a subsequence $(u_j)$ (denoted by $u_j$ again)
$P(u_j) \to P(u)$ as $j \to \infty$. In fact, by the compact embedding $H^s \hookrightarrow L^p$, $2 < p < \frac{2n}{n-2s}$, we can find a subsequence $(u_j)$ such that $u_j \to u$ in $L^p$. So, it is clear that  $\int F(x, |u_j|)\,dx \to \int F(x, |u|)\,dx$ for $a \in L^\infty$. Now it remains to show that $\int V|u_j|^2\,dx \to \int V|u|^2\,dx$.
By \eqref{cond-V} we deduce that for any $\varepsilon > 0$ there is an $R > 0$ such that
\begin{align*}
\int_{|x| > R}V|u_j|^2\,dx \le C\int_{|x| > R} V(x)|x|^{-(n-2s)}\,dx\|(-\Delta)^\frac s2 u_j\|_{L^2}^2 \le \frac \varepsilon2.
\end{align*}
Since $V \in L^\infty_{loc}$ and the embedding $H^s \hookrightarrow L^2_{loc}$ is compact, (up to a subsequence) there exists $j_0 > 0$ such that
$$
|\int_{|x| \le R}V|u_j|^2\,dx -  \int_{|x| \le R}V|u|^2\,dx| \le 2\sqrt c\|V\|_{L^\infty(|x| \le R)}\|u_j - u\|_{L^2(|x| \le R)} \le \frac\varepsilon2
$$
if $j \ge j_0$. This completes the proof of Proposition \ref{orb2}.

\section{Proof of Theorems \ref{stability1}, \ref{stability2}}\label{orb}

Since the proofs of Theorems \ref{stability1}, \ref{stability2} are quite the same, we only consider the proof of Theorem \ref{stability1}.
The proof proceeds by contradiction. Suppose that $\mathcal O_\mu$ is not stable, then either $\mathcal O_\mu$ is empty or there exist $w \in \mathcal O_\mu$ and a sequence $\p_0^j \in H^s$ such that
$$
\|\p_0^j - w\|_{H^s} \to 0 \;\;\mbox{as}\;\;j \to \infty
$$
but
\begin{align}\label{cont1}
\inf_{v \in \mathcal O_\mu} \|\p^j(t_j, \cdot) - v\|_{H^s} \ge \varepsilon_0
\end{align}
for some sequence $t_j \in \mathbb R$ and $\varepsilon_0$, where $\p^j(t, \cdot)$ is the solution of \eqref{main eqn} corresponding to the initial data $\p_0^j$.
Let $w_j = \p^j(t_j, \cdot)$. Since $w \in S_\mu$ and $J(w) = I_\mu$, it follows from the continuity of $L^2$ norm and $J$ in $H^s$ that
$$
\|\p_0^j\|_{L^2} \to \mu \;\;\mbox{and}\;\; J(\p_0^j) \to I_\mu.
$$
Thus we deduce from the conservation laws that
$$
\|w_j\|_{L^2} = \|\p_0^j\|_{L^2} \to \mu,\quad J(w_j) = J(\p_0^j) \to I_\mu.
$$
Therefore if $w_j$ has a subsequence converging to an element $w \in H^s$ such that $\|w\|_{L^2} = \mu$ and $J(w) = I_\mu$.
This shows that $w \in \mathcal O_\mu$ but
$$
\inf_{v \in \mathcal O_\mu} \|\p^j(t_j, \cdot) - v\|_{H^s} \le \|w_j - w\|_{H^s},
$$
which contradicts \eqref{cont1}. Since $\mathcal O_\mu$ is not empty, to show the orbital stability of $\mathcal O_\mu$ one has to prove that any sequence $(w_j) \subset H^s$ with
\begin{align}\label{cont2}
\|w_j\|_{L^2} \to \mu \;\;\mbox{and}\;\; J(w_j) \to I_\mu
\end{align}
is relatively compact in $H^s$.
Since $I_\mu$ is continuous w.r.t $\mu \in (0, \infty)$ and $\ell \le \frac{4s}n$, by the arguments in the proof of Proposition \ref{orb1} we may assume that $(w_j)$ is bounded in $H^s$ and also verify from \eqref{comparable} that by passing to a subsequence there exists $w \in H^s$ such that
\begin{align}\label{cont3}
w_j \rightharpoonup w \;\;\mbox{in}\;\;H^s\;\;\mbox{and}\;\; \lim_{j\to \infty}\|(-\Delta)^\frac s2 w_j\|_{L^2} = \|(-\Delta)^\frac s2 w\|_{L^2}.
\end{align}
This implies $w_j \to w$ in $H^s$ and thus the relative compactness.

\section{Uniqueness and well-posedness}\label{uniq}
In this section we show the existence of weak solutions and its uniqueness.

\subsection{Uniqueness of weak solution}

We first consider the existence of weak solutions to \eqref{main eqn}.

\begin{proposition}\label{ws1}
Let $n \ge 1$, $0 < s < 1$ and $0 < \ell < s^*-2$, $s^* = \frac{2n}{n-2s}$ if $ n > 2s$ and $s^* = \infty$ if $n = 1$ and $\frac12 \le s < 1$. Let $V \in L_{loc}^{p_1} + L^{p_2}(|x| > 1)$ for $\frac{n}{2s}< p_1, p_2 \le \infty$. Let $f$ satisfy \eqref{cond-f} and \eqref{cond-f2}  with $a, b \in L_{loc}^{q_1} + L^{q_2}(|x| > 1)$ for $q_0 < q_1, q_2 \le \infty$, where
$q_0 = \frac{2n}{2n-(\ell+2)(n-2s)}$ if $n > 2s$,  and $q_0 = 1$ if $n = 1$ and $\frac12 \le s < 1$.
 Then there exists a weak solution $\p$ such that \begin{eqnarray*} &\p \in L^\infty(-T_{min}, T_{max} ; H^s) \cap W^{1, \infty}(-T_{min}, T_{max} ; H^{-s}),\\ &M(\p(t)) = M(\varphi),\;\;J(\p(t)) \le J(\varphi)\end{eqnarray*} for all $t \in (-T_{min}, T_{max})$, where $(-T_{min}, T_{max})$ is the maximal existence time interval of $\p$ for given initial data $\varphi$.
\end{proposition}
\begin{proof}[Proof of Proposition \ref{ws1}]
To show the existence of weak solutions we follow the standard regularizing argument (for instance see \cite{caz}). For this purpose we have only to verify that $\|N(x, \p) - N(x, \Psi)\|_{H^{-s}} \le C(K)\|\p-\Psi\|_{H^s}$, provided $\|\p\|_{H^s} + \|\Psi\|_{H^s} \le K$.
In fact, since $p_i > \frac{n}{2s}, i = 1, 2$, we can always find $r_i, l_i \in [2, s^*)$ such that
$$
\|V(\p-\Psi)\|_{L^{l_1'} + L^{l_2'}} \le \|V\|_{L^{p_1} + L^{p_2}} \|\p-\Psi\|_{L^{r_1}+L^{r_2}}
$$
for $\frac1{r_i} = 1 - \frac1{l_i} - \frac1{p_i}$.
If $n > 2s$ and $q_0 < q_i \le \infty$,
\begin{align*}
&\|f(x,\p) - f(x,\Psi)\|_{L^{\widetilde l_1'} + L^{\widetilde l_2'}}\\
 &\le C\|b\|_{L^{q_1} + L^{q_2}}(\|\p\|_{L^{\widetilde r_1} + L^{\widetilde r_2}}^\ell + \|\Psi\|_{L^{\widetilde r_1} + L^{\widetilde r_2}}^\ell)\|\p-\Psi\|_{L^{\widetilde r_1}+L^{\widetilde r_2}}\\
 &\le C\|b\|_{L^{q_1} + L^{q_2}}(2K^\ell)\|\p-\Psi\|_{L^{\widetilde r_1}+L^{\widetilde r_2}}
\end{align*}
for $\frac{\ell+1}{\widetilde r_i} = 1- \frac1{q_i}- \frac1{\widetilde l_i}$. Here we used the Sobolev embedding $H^s \hookrightarrow L^{\widetilde r_i}$.
If $n = 1$ and $\frac12 < s < 1$, for any $1 < q_1, q_2 \le \infty$ we can find $ \widetilde r_i, \widetilde l_i \in [2, \infty)$ such that
$\frac1{\widetilde l_i} = \frac1q_i + \frac1{\widetilde r_i}$. Thus from the embedding $H^s \hookrightarrow L^\infty$ we have
\begin{align*}
&\|f(x,\p) - f(x,\Psi)\|_{L^{\widetilde l_1'} + L^{\widetilde l_2'}}\\
 &\le C\|b\|_{L^{q_1} + L^{q_2}}(\|\p\|_{L^\infty}^\ell + \|\Psi\|_{ L^\infty}^\ell)\|\p-\Psi\|_{L^{\widetilde r_1}+L^{\widetilde r_2}}\\
 &\le C\|b\|_{L^{q_1} + L^{q_2}}(2K^\ell)\|\p-\Psi\|_{L^{\widetilde r_1}+L^{\widetilde r_2}}.
\end{align*}
If $n = 1$ and $s = \frac12$, then for any $1 < q_1, q_2 \le \infty$, we can find $ \widetilde r_i, \widetilde l_i \in [2, \infty)$ and $\widetilde p_i \gg 1$ such that $\frac1{\widetilde l_i} = \frac1q_i + \frac{\ell}{\widetilde p_i} + \frac1{\widetilde r_i}$. So, we have
\begin{align*}
&\|f(x,\p) - f(x,\Psi)\|_{L^{\widetilde l_1'} + L^{\widetilde l_2'}}\\
 &\le C\|b\|_{L^{q_1} + L^{q_2}}(\|\p\|_{L^{\widetilde p_1} + L^{\widetilde p_2}}^\ell + \|\Psi\|_{L^{\widetilde p_1} + L^{\widetilde p_2}}^\ell)\|\p-\Psi\|_{L^{\widetilde r_1}+L^{\widetilde r_2}}\\
 &\le C\|b\|_{L^{q_1} + L^{q_2}}(2K^\ell)\|\p-\Psi\|_{L^{\widetilde r_1}+L^{\widetilde r_2}}.
\end{align*}
This proves Proposition \ref{ws1}.
\end{proof}

\subsubsection{1-d uniqueness}
\begin{proposition}\label{1dunique}
Let $n = 1$, $\frac12 < s < 1$, and $0 < \ell < \infty$. Suppose that $V$ and $f$ satisfy the conditions of Proposition \ref{ws1} and further $V, b \in L^\infty$.  Then the $H^s$-weak solution to \eqref{main eqn} is unique.
\end{proposition}
\begin{proof}[Proof of Proposition \ref{1dunique}]
The solution $\p$ constructed in Proposition \ref{ws1} satisfies the integral equation
\begin{align}\label{int eqn}
\p(t) = U(t)\varphi - i\int_0^t U(t-t')N(\cdot, \p(t'))\,dt'\;\;\mbox{a.e.}\;t \in (-T_{min}, T_{max}),
\end{align}
where $U(t) = e^{it(-\Delta)^s}$. 
Let $\Psi$ be another weak solution of \eqref{main eqn} with the same initial data as $\p$ on the interval $[-T_1, T_2] \subset (-T_{min}, T_{max})$. Assume that $\|\p\|_{L^\infty(-T_1, T_2; H^s)} + \|\Psi\|_{L^\infty(-T_1, T_2; H^s)} \le K$. Then for any interval $[-t_1, t_2] \in [T_1, T_2]$ we have
\begin{align*}
&\|\p-\Psi\|_{L^\infty(-t_1, t_2; L^2)}\\
 &\le C\int_{-t_1}^{t_2}(\|V(\p-\Psi)\|_{L^2} + \|b(|\p|^\ell + |\Psi|^\ell)|\p-\Psi|\|_{L^2})\,dt'\\
&\le C\int_{-t_1}^{t_2}(\|V\|_{L^\infty} + \|b\|_{L^\infty}(\|\p\|_{L^\infty}^\ell + \|\Psi\|_{L^\infty}^\ell))\|\p-\Psi\|_{L^2}\,dt'\\
&\le C(t_1 + t_2)(\|V\|_{L^\infty} + \|b\|_{L^\infty}(2K^\ell))\|\p-\Psi\|_{L^\infty(-t_1, t_2; L^2)}.
\end{align*}
So $\p = \Psi$ on $[-t_1, t_2]$ for sufficiently small $t_1, t_2$. Let $I = (-a , b)$ be the maximal interval  of $[-T_1, T_2]$ with $\|\p - \Psi\|_{L^\infty(-c, d; L^2)} = 0$ for $c < a, d < b$. Suppose that $a < T_1$ or $b < T_2$. Without loss of generality, we may assume that $a < T_1$ and $b < T_2$. Then for a small $\varepsilon > 0$ we can find $ a < t_1 < T_1, b < t_2 < T_2$ such that
\begin{align*}
&\|\p-\Psi\|_{L^\infty(-t_1, t_2; L^2)}\\
&\le C\int_{-t_1}^{t_2}(\|V\|_{L^\infty} + \|b\|_{L^\infty}(\|\p\|_{L^\infty}^\ell + \|\Psi\|_{L^\infty}^\ell))\|\p-\Psi\|_{L^2}\,dt'\\
&\le C(\|V\|_{L^\infty} + \|b\|_{L^\infty}(2K^\ell))\int_{-t_1, t_2}\|\p-\Psi\|_{ L^2}\,dt'\\
&= C(\|V\|_{L^\infty} + \|b\|_{L^\infty}(2K^\ell))(t_1+t_2-a-b)\|\p-\Psi\|_{L^\infty(-t_1, t_2; L^2)}\\
&\le (1-\varepsilon)\|\p-\Psi\|_{L^\infty(-t_1, t_2; L^2)}.
\end{align*}
This contradicts the maximality of $I$. Thus $I = [-T_1, T_2]$. Since $[-T_1, T_2]$ is arbitrarily taken in $(-T_{min}, T_{max})$, we finally get the whole uniqueness.
\end{proof}

\subsubsection{Conditional uniqueness for $n \ge 2$}
The weak solution can be shown to be unique under an weighted integrability condition. For this purpose we introduce a mixed norm $\|h\|_{L_\rho^m L_\sigma^{\widetilde m}}$, defined by $(\int_0^\infty (\int_{S^1}|h(\rho\sigma)|^{\widetilde m}\,d\sigma)^\frac m{\widetilde m}\,\rho^{n-1} d\rho)^\frac1m$ for $1\le m, \widetilde m < \infty$. The case $m = \infty$ or $\widetilde m =\infty$ can be defined in the usual way. We set $\ell_0 = \frac{2s-1 }{2s(1-s)}$ if $n = 2$ and $\frac{n(2s-1)}{(n-2s)(n-1)}$ if $n \ge 3$, and set $\delta = \frac{n+2s}{q} -\frac n2$. Then we have the following.
\begin{proposition}\label{unique}
Let $n \ge 2$, $\frac12 < s < 1$, and $0 < \ell < \ell_0$ if $n = 2$. Suppose that $V$ and $f$ satisfy the conditions of Proposition \ref{ws1} and further that $|x|^{\delta} V \in L_\rho^{m_1}L_\sigma^{\widetilde m_1}$, $|x|^{\delta} b \in L_\rho^{m_2}L_\sigma^{\widetilde m_2}$ for $$\frac1{m_1} = \frac12- \frac1{q}, \quad \frac1{\widetilde m_1} = \frac{\frac{2s}{q} - \frac12}{n-1}$$  and $$\frac 1{m_2} = \frac12- \frac1{q} - \frac{\ell(n-2s)}{2n}, \quad \frac1{\widetilde m_2} = \frac{\frac{2s}{q} - \frac12}{n-1} - \frac{\ell(n-2s)}{2n}$$ with $\frac12 + \frac{(n-1)\ell(n-2s)}{2n} \le \frac{2s}{q} \le s $.  Then the $H^s$-weak solution to \eqref{main eqn} is unique.
\end{proposition}
\begin{proof}[Proof of Proposition \ref{unique}]


For the uniqueness we will use the following weighted Strichartz estimate (see for instance Lemma 6.2 of \cite{chho} and Lemma 2 of \cite{chhwoz}).
\begin{lemma}
Let $n \ge 2$ and $2 \le  q < 4s$. Then we have
\begin{align}\label{w-str}
\||x|^{-\delta}  U(\cdot)\varphi\|_{L^q(-t_1, t_2; L_\rho^qL_\sigma^{\widetilde q})} \le C\|\varphi\|_{L^2},
\end{align}
where $\delta = \frac{n+2s}{q}-\frac n2$, $\frac1{\widetilde q} = \frac12 - \frac1{n-1}\left(\frac{2s}{q} - \frac12\right) $ and $C$ is independent of $t_1, t_2$.
\end{lemma}
In \cite{chho} it was shown that $$\||x|^{-\delta} D_\sigma^{\frac{2s}q-\frac12} U(\cdot)\varphi\|_{L^q(-t_1, t_2; L_\rho^qL_\sigma^2)} \le C\|\varphi\|_{L^2}.$$ The inequality \eqref{w-str} can be derived by Sobolev embedding on the unit sphere. Here $D_\sigma = \sqrt{1 - \Delta_\sigma}$, $\Delta_\sigma$ is the Laplace-Beltrami operator on the unit sphere.



We first consider the 2-d case. From \eqref{w-str} one can readily deduce that
\begin{align}\label{w-str-inhomo}
\||x|^{-\delta}  \int_0^t U(t-t')g(t')\|_{L^{q}(-t_1,t_2; L_\rho^{q} L_\sigma^{\widetilde q})} \lesssim \|g\|_{L^1(-t_1,t_2; L^2)}
\end{align}
Set $g = N(x, \p) - N(x, \Psi)$. Then from \eqref{int eqn} we have
\begin{align*}
&\||x|^{-\delta} (\p - \Psi)\|_{L^{q}_{(-t_1,t_2)}L_\rho^{q} L_\sigma^{\widetilde q}}\\
&\lesssim \int_{-t_1}^{t_2}\|N(\cdot, \p) - N(\cdot, \Psi)\|_{L^2}\,dt'\\
&\lesssim \int_{-t_1}^{t_2}\|V(\p - \Psi)\|_{L^2}\,dt' + \int_{-t_1}^{t_2}\|b(\cdot)(|\p|^\ell + |\Psi|^\ell)|\p - \Psi|\|_{L^2}\,dt'.
\end{align*}
By H\"{o}lder's inequality with $\frac12 = \frac1{q} + \frac1{m_1} = \frac1{\widetilde q} + \frac1{\widetilde m_1}$ and $\frac12 = \frac1{m_2} + \frac{\ell(n-2s)}{2n} + \frac1q = \frac1{\widetilde m_2} + \frac{\ell(n-2s)}{2n} + \frac1{\widetilde q}$ we have
 \begin{align*}
&\||x|^{-\delta} (\p - \Psi)\|_{L^{q}_{(-t_1,t_2)}L_\rho^{q} L_\sigma^{\widetilde q}} \\
&\lesssim \||x|^{\delta}V\|_{L_\rho^{m_1}L_\sigma^{\widetilde m_1}}\int_{-t_1}^{t_2} \||x|^{-\delta_1}(\p-\Psi)\|_{L_\rho^{q}L_\sigma^{\widetilde q}}\,dt'\\
 &\;\;+ \int_{-t_1}^{t_2}\||\cdot|^{\delta}b\|_{L_\rho^{m_2}L_\sigma^{\widetilde m_2}}(\|\p\|_{L^\frac{2}{1-s}}^\ell + \|\Psi\|_{L^\frac{2}{1-s}}^\ell)\||x|^{-\delta}(\p - \Psi)\|_{L_\rho^qL_\sigma^{\widetilde q}}\,dt'.
\end{align*}
Our $q$ and $\ell$ for $n = 2$ guarantee the well-definedness of the H\"{o}lder exponents $m_1, \widetilde m_1, m_2, \widetilde m_2$.

If $(-t_1, t_2) \subset [-T_1, T_2]$ and $\|\p\|_{L^\infty(-T_1, T_2; H^s)} + \|\Psi\|_{L^\infty(-T_1, T_2; H^s)} \le K$, then by Sobolev's and H\"{o}lder's inequalities we have
 \begin{align*}
&\||x|^{-\delta} (\p - \Psi)\|_{L^q_{(-t_1,t_2)}L_\rho^q L_\sigma^{\widetilde q}} \le C_{V,f, K} (t_1+t_2)^{1-\frac1q}\||x|^{-\delta} (\p - \Psi)\|_{L^q_{(-t_1,t_2)}L_\rho^q L_\sigma^{\widetilde q}},
\end{align*}
where $C_{V, f, K} = C\left(\|\cdot|^\delta V\|_{L_\rho^{m_1}L_\sigma^{\widetilde m_1}} + \||\cdot|^{\delta}b\|_{L_\rho^{m_2}L_\sigma^{\widetilde m_2}}(2K^\ell)\right)$. So $\p = \Psi$ on $[-t_1, t_2]$ for sufficiently small $t_1, t_2$. By the same argument as in 1-d case we can extend this uniqueness to $(-T_{min}, T_{max})$.

We can proceed with the almost same way as the proof of uniqueness for high-d case. The only difference is the range of $\ell$. For the proof we need $\frac1{m_2}, \frac1{\widetilde m_2} \ge 0$, for which we must have
$$
\ell \le \frac{n(2s-1)}{(n-1)(n-2s)},\quad \ell < \frac{n(2s-1)}{2s(n-2s)},\quad n \ge 2,
$$
respectively. In 2-d, the former is bigger than the latter and vice versa in high-d.
\end{proof}

\subsubsection{Unconditional uniqueness}
If we restrict $q$ in Proposition \ref{unique}, then we can get the unconditional uniqueness as follows.
\begin{corollary}\label{2dcor}
If $n = 2$, $\frac1{4s} < \frac1q \le \min\left(\frac12 - \frac{\ell(1-s)}{s}, \frac1{1+2s}\right)$, then the uniqueness as in Proposition \ref{unique} occurs in $C((-T_{min}, T_{max}); H^s)$.
\end{corollary}
\begin{proof}[Proof of Corollary \ref{2dcor}]
In view of the proof of Proposition \ref{unique}, we have only to show that $$C((-T_{min}, T_{max}); H^s) \subset L_{loc}^q(-T_{min}, T_{max};|x|^{\frac{2+2s}{q}-1} L_\rho^q L_\sigma^{\widetilde q}),$$ where $\frac1{\widetilde q} = \frac12 - \frac1{n-1}\left(\frac{2s}{q} - \frac12\right)$. Since $q \ge 1 + 2s$,  $\widetilde q \le q$ and thus
$$ L_{loc}^q(-T_{min}, T_{max}; |x|^{\frac{2+2s}{q}-1} L_x^q) \subset L_{loc}^q(-T_{min}, T_{max}; |x|^{\frac{2+2s}{q}-1}L_\rho^q L_\sigma^{\widetilde q}).$$
Using Hardy inequality\footnote{Such inequality can be shown by the interpolation between the estimates $\||x|^{-\alpha}\,u\|_{L^2} \lesssim \|u\|_{\dot H^\alpha}$ and $\|u\|_{BMO} \lesssim \|u\|_{\dot H^\frac n2}$.} that $\||x|^{-\alpha}\,u\|_{L^p} \lesssim \|u\|_{\dot H^{\alpha + n(\frac12-\frac1p)}}$ for $0 < \alpha < \frac np$ and $2 \le p < \infty$, since $1 - \frac{2+2s}{q} < 0$ and $q \ge 2 > 2s$ for the above $q$, we get
$$
\||x|^{1-\frac{2+2s}{q}} u\|_{L_x^q} \lesssim \|u\|_{\dot H^\frac{2s}{q}} \lesssim \|u\|_{H^s}.
$$
By this we deduce that $H^s \subset |x|^{\frac{2+2s}{q}-1}L_x^q$. This completes the proof of corollary.
\end{proof}

By exactly the same way, we have the following.
\begin{corollary}\label{ndcor}
If $n \ge 3$, $\max\left(\frac12,  \frac{ns}{n+2s}, \frac12 + \frac{(n-1)\ell(n-2s)}{2n}\right) < \frac{2s}{q} \le \frac{ns}{n+2s-1}$ and $0 < \ell \le \frac{n(2s-1)}{(n-2s)(n+2s-1)}$, then the uniqueness as in Proposition \ref{unique} occurs in $C((-T_{min}, T_{max}); H^s)$.
\end{corollary}

\subsection{Well-posedness}\label{wp}
By using the argument of \cite{caz} one can show that the uniqueness implies actually well-posedness and conservation laws:
\begin{align*}
&\bullet\; \p \in C(-T_{min}, T_{max}; H^s) \cap C^1(-T_{min}, T_{max}; H^{-s}), \\
&\bullet\; \p\;\;\mbox{depends continuously on}\;\; \varphi\;\;\mbox{in}\;\;H^s,\\
&\bullet\; M(\p(t)) = M(\varphi)\;\;\mbox{ and}\;\;J(\p(t)) = J(\varphi)\;\; \forall\; t \in (-T_{min}, T_{max}).
\end{align*}
We leave the details to the readers. In this section we remark on the global well-posedness.

\subsubsection{Remarks on global well-posedness}
We discussed that the conditions on $V$ and $f$ of Propositions \ref{ws1}-\ref{unique}  give the uniqueness and local well-posedness of \eqref{main eqn}. In this section we study some conditions guaranteeing the global well-posedness.


\begin{enumerate}
\item If $f(x, |\tau|) \le 0$ for all $x$ and $\tau$, and $\p$ is the unique solution to \eqref{main eqn}, then
\begin{align*}
J(\varphi) = J(\p) \ge \frac12\|(\Delta)^\frac s2 \p\|_{L^2}^2 - \|V\|_{L^{p_1} + L^{p_2}}\|\p\|_{L^{2p_1'} + L^{2p_2'}}.
\end{align*}
Since $2 < 2p_i' < s^*$, from Gagiliardo-Nirenberg inequality it follows that for some $0 < \theta_i < 1$
\begin{eqnarray*}
&J(\varphi) \ge \frac12\|(\Delta)^\frac s2 \p\|_{L^2}^2\\
 &\quad - C\|V\|_{L^{p_1} + L^{p_2}}(\|\varphi\|_{L^2}^{2(1-\theta_1)}\|(-\Delta)^\frac s2\p\|_{L^2}^{2\theta_1} + \|\varphi\|_{L^2}^{2(1-\theta_2)}\|(-\Delta)^\frac s2\p\|_{L^2}^{2\theta_2}).
\end{eqnarray*}
By Young's inequality we get
$\|\p(t)\|_{H^s}^2 \le C_V\|\varphi\|_{L^2}^2 + 4J(\varphi)$ for all $t \in (-T_{min}, T_{max})$. The continuity argument implies the global well-posedness that $T_{min} = T_{max} = \infty$.

\item If $\ell < \frac{4s}{n}$ and $q_i > \frac{2n}{4s-n\ell} (> \frac{2n}{2n-(\ell+2)(n-2s)})$, then Gagliardo-Nirenberg inequality gives the uniform bound of $\|\p(t)\|_{H^s}$. More precisely,
\begin{align*}
J(\varphi) = J(\p) &\ge \frac12\|(-\Delta)^\frac s2 \p\|_{L^2}^2 - \|V\|_{L^{p_1} + L^{p_2}}\|\p\|_{L^{2p_1'} + L^{2p_2'}}^2\\
 &\qquad - \|a\|_{L^{q_1} + L^{q_2}}\|\p\|_{L^{(\ell+2)q_1'} + L^{(\ell+2)q_2'}}^{\ell+2}.
\end{align*}
Since $q_i > \frac{2n}{4s-n\ell}$, $2 < (\ell+2)q_i' < \frac{2n}{n-2s}$. As above we get $\|\p(t)\|_{H^s}^2 \le C_{V, a}\|\varphi\|_{L^2}^2 + 4J(\varphi)$ for all $t \in (-T_{min}, T_{max})$ and thus global well-posedness.

\item If $n = 1$, $\frac12 < s < 1$, $V = 0$ and $f(x, \p) = \lambda |\p|^2\p$, then in \cite{guhuo} the authors showed global well-posedness in $L^2$ by using Bourgain space argument. For another global well-posedness we refer the reader to \cite{iopu}, where a problem with $s = \frac14$ and combined cubic nonlinearity is treated.

\item If $\ell = \frac{4s}{n}$, then since
\begin{align*}
J(\varphi) = J(\p) &\ge \frac12\|(-\Delta)^\frac s2 \p\|_{L^2}^2 - \|V\|_{L^{p_1} + L^{p_2}}\|\p\|_{L^{2p_1'} + L^{2p_2'}}^2\\
 &\qquad - \|a\|_{L^\infty}\|\varphi\|_{L^2}^\frac{4s}n\|(-\Delta)^\frac s2\p\|_{L^2}^2,
\end{align*}
we have $\|\p(t)\|_{H^s}^2 \le C\|\varphi\|_{L^2}^2 + 4J(\varphi)$ for $\|a\|_{L^\infty}\|\varphi\|_{L^2}^\frac{4s}n < \frac18$.
\item If $q$ and $\ell$ satisfies the condition of Corollary \ref{2dcor} and \ref{ndcor}, then the well-posedness is unconditional.
\end{enumerate}

\subsubsection{Well-posedness of radial solutions}
From now on we consider the well-posedness of radial solutions to \eqref{main eqn} when $\ell \le \frac{4s}{n-2s}$.
In \cite{chho} the authors considered the well-posedness for Hartree type nonlinearity by using various Strichartz estimates. Indeed, they utilized weighted or angularly regular Strichartz estimate to control the Hartree type nonlinearity. However, if the power type nonlinearity $f$ is involved, then the situation is quite different. It is not easy to handle angular regularity for which we need a high regularity of $f$. To avoid this we assume the radial symmetry of $f$ and initial data.

Let us introduce radial Strichartz estimate of $U(t)$ (see \cite{cholee}): for $\frac{n}{2n-1} \le s < 1$, $2 \le q \le \infty, 2 \le r < \infty$ with $\frac{2s}{q} + \frac nr = \frac n2$ and $(q, r) \neq (2, \frac{4n-2}{2n-3})$
\begin{align}
\|U(\cdot)\varphi\|_{L^q(-T_1, T_2; L_x^r)} \lesssim \|\varphi\|_{L^2}.\label{homo}
\end{align}
We call such pair $(q, r)$ $s$-admissible one. The constant involved in \eqref{homo} is independent of $T_1$, $T_2$. The estimate \eqref{homo} can be extended to Besov type as follows:
\begin{align}
\|U(\cdot)\varphi\|_{L^q(-T_1, T_2; B_{r}^s)} \lesssim \|\varphi\|_{H^s}.\label{homo-bes}
\end{align}
Here $B_r^s = B_{r, 2}^s$ is the inhomogeneous Besov space. Using Christ-Kiselev lemma we get the inhomogeneous Strichartz estimates:
Let $(q, r)$ and $(\widetilde q, \widetilde r)$ be $s$-admissible pairs with $q > \widetilde q'$. Then
\begin{align}
\|\int_0^t U(t-t')g(t')\,dt'\|_{L^q(-T_1, T_2; B_{r}^s)} \lesssim \|g\|_{L_T^{\widetilde q'}B_{\widetilde r'}^s}.\label{inhomo}
\end{align}

Under the fractional and power type setting, an alternative Besov norm is useful, which is stated as follows: for $0 < s < 1$, $1 \le r < \infty$
\begin{align}\label{bes}
\|g\|_{B_{r}^s} \sim \|g\|_{L^r} + \left(\int_0^\infty (\alpha^{-s}\sup_{|y|\le \alpha}\|g(\cdot+y)-g(\cdot)\|_{L^r})^2\,\frac{d\alpha}{\alpha}\right)^\frac12.
\end{align}

The following is the local well-posedness result.
\begin{proposition}\label{lwp-radial}
Suppose that $V(x) = V(|x|) \in L^\infty$, $(-\Delta)^\frac s2V \in L^\frac ns$, $f(x, z) = f(|x|, z)$, $a, b \in L^\infty$ and $\varphi(x) = \varphi(|x|) \in H^s$. Let $$
r_0 = \frac{n(\ell+2)}{n+s\ell}, \quad q_0 = \frac{4s(\ell+2)}{\ell(n-2s)}$$
for $\frac{n}{2n-1} \le s < 1$, and $0 < \ell \le \frac{4s}{n-2s}$. Then there exists $T_1, T_2 > 0$ such that \eqref{main eqn} has a unique radial solution $u \in C([-T_1,T_2]; H^s) \cap L^{q_0}(-T_1, T_2; B_{r_0}^s)$.
\end{proposition}
The pair $(q_0, r_0)$ is $s$-admissible one and $q_0 > 2$.

\begin{proof}[Proof of Proposition \ref{lwp-radial}]
For simplicity we only consider the well-posednss on $[0, T]$.
Let $(X_T^\rho, d_X)$ be a metric space with metric $d_X$
defined by
\begin{align*}
&X_T^\rho = \{\p \in L_T^\infty H^s\, \cap\, L_T^{q_0}B_{r_0}^s :\;\;\p\;\mbox{is radial and}\;
\|\p\|_{L_T^\infty H^s \,\cap\, L_T^{q_0}B_{r_0}^s} \le \rho\},\\
&\qquad\qquad\qquad\qquad\qquad d_X(\p, \Psi) = \|\p -
\Psi\|_{L_T^\infty L^2 \cap L_T^{q_0}L^{r_0}}.
\end{align*}
$L_T^q\mathcal B$ denotes $L_t^q([0,T]; \mathcal B)$ for some positive $T$ and Banach space $\mathcal B$.
Since $H^s$ and $B_{r_0}^s$ are reflexive Banach space, one can readily show that $X_T^\rho$ is complete.
We define a mapping $\mathcal N$ on $X_T^\rho$ by
\begin{align}\label{nonlinear func}
\mathcal N(\p)(t) = U(t)\varphi - i\int_0^t U(t-t')[N(\cdot, \p)](t')\,dt'.
\end{align}
We use the standard contraction mapping argument. For any $\p \in X_T^\rho$ we have from \eqref{homo-bes} and \eqref{inhomo} with $q_0, r_0$ that
\begin{align}\label{nonlinear}
\|\mathcal N(\p)\|_{L_T^\infty H^s \cap L_T^{q_0}B_{r_0}^s} \lesssim \|U(\cdot)\varphi\|_{L_T^{q_0}B_{r_0}^s} + \|V\p\|_{L_T^1 H^s} + \|f(\cdot, \p)\|_{L_T^{q_0'} B_{r_0'}^s}.
\end{align}

From the fractional Leibniz rule, we have
\begin{align}\begin{aligned}\label{est-V}
\|V\p\|_{L_T^1 H^s} &\lesssim T(\|V\|_{L^\infty}\|\p\|_{L_T^\infty H^s} + \|(-\Delta)^\frac s2 V\|_{L^\frac ns}\|\p\|_{L_T^\infty L^\frac{2n}{n-2s}})\\
&\lesssim T(\|V\|_{L^\infty} + \|(-\Delta)^\frac s2 V\|_{L^\frac ns})\|\p\|_{L_T^\infty H^s}.
\end{aligned}\end{align}

On the other hand, since $\frac1{r_0'} = \frac{\ell}{r} + \frac1{r_0}$ for $\frac1r  = \frac{n-2s}{n(\ell+2)}$ which equals $\frac1{r_0} - \frac sn$, from the condition \eqref{cond-f} and Sobolev embedding $B_{r_0}^s \hookrightarrow L^r$ it follows that
$$
\|f(\cdot, \p)\|_{L_x^{r_0'}} \lesssim \|\p\|_{L_x^r}^\ell\|\p\|_{L_x^{r_0}} \lesssim \|\p\|_{B_{r_0}^s}^{\ell+1}.
$$
From \eqref{cond-f2} we have
\begin{align*}
&|f(x+y, \p(x+y)) - f(x, \p(x))| \\
&\lesssim \min(1, |y|)|\p(x+y)|^{\ell+1} + (|\p(x+y)|^\ell + |\p(x)|^\ell)|\p(x+y) - \p(x)|.
\end{align*}
So, we get as above
\begin{align*}
&\|f(\cdot+y, \p(\cdot+y)) - f(\cdot, \p(\cdot))\|_{L^{r_0'}}\\
& \lesssim \min(1, |y|)\|\p\|_{B_{r_0}^s}^{\ell+1} + \|\p\|_{B_{r_0}^s}^\ell\|\p(\cdot+y)-\p(\cdot)\|_{L^{r_0}}.
\end{align*}
Thus we get from \eqref{bes}
\begin{align}\begin{aligned}\label{f-bes}
\|f(\cdot, \p)\|_{B_{r_0'}^s} \lesssim \|\p\|_{B_{r_0}^s}^{\ell+1}(1+ \left(\int_0^\infty [\alpha^{-s}\min(1, \alpha)]^2\frac{d\alpha}\alpha\right)^\frac12) \lesssim \|\p\|_{B_{r_0}^s}^{\ell+1}.
\end{aligned}\end{align}


Now let us turn to the nonlinear estimate \eqref{nonlinear}. We take H\"{o}lder's inequality in $t$-variable with
$$
\frac1{q_0'} = \frac{\ell+1}{q_0} + \frac1{q_1}.
$$
From the condition of $\ell$ we have $1/{q_1} \ge 0$, and $1/q_1 = 0$ when $\ell = \frac{4s}{n-2s}$. Thus we get
\begin{align*}
\|\mathcal N(\p)\|_{L_T^\infty H^s \cap L_T^{q_0}B_{r_0}^s} &\le \|U(\cdot)\varphi\|_{L_T^{q_0}B_{r_0}^s} + CT\|\p\|_{L_T^\infty H^s} +  CT^\frac1{q_1} \|\p\|_{L_T^{q_0}B_{r_0}^s}^{\ell+1}\\
&\lesssim \|U(\cdot)\varphi\|_{L_T^{q_0}B_{r_0}^s} + CT\rho + CT^\frac1{q_1} \rho^{\ell+1}.
\end{align*}
If $\ell < \frac{4s}{n-2s}$, then since $\|U(\cdot)\varphi\|_{L_T^{q_0}B_{r_0}^s} \lesssim \|\varphi\|_{H^s}$, $T$ can be chosen to be dependent only on $C$ and $\|\varphi\|_{H^s}$ to guarantee $\mathcal N(\p) \in X_T^\rho$. If $\ell = \frac{4s}{n-2s}$, then we first choose $T, \rho$ such that $C(T\rho + T^\frac1{q_1}\rho^{\ell+1}) \le \rho/2$ and then choose smaller $T$ such that $\|U(\cdot)\varphi\|_{L_T^{q_0}B_{r_0}^s} \le \frac\rho2$, which means $\mathcal N(\p) \in X_T^\rho$.

\newcommand{\ps}{\Psi}
Now we show that $\mathcal N$ is a Lipschitz map for sufficiently
small $T$. Let $\p, \ps \in X_T^\rho$. Then from the same estimates as above we have
\begin{align*}
d_X(\mathcal N(\p), \mathcal N(\ps))&\le C\|V(\p-\Psi)\|_{L_T^1H^s } + C \|f(\cdot, \p) - f(\cdot, \ps)\|_{L_T^{q_0'}L^{r_0'}}\\
&\le CT(\|V\|_{L^\infty} + \|(-\Delta)^\frac s2 V\|_{L^\frac ns})\|\p - \Psi\|_{L_T^\infty H^s}\\
&\qquad\quad + C T^\frac1{q_1}(\|\p\|_{L_T^{q_0}B_{r_0}^s}^{\ell} + \|\ps\|_{L_T^{q_0}B_{r_0}^s}^\ell)\|\p-\ps\|_{L_T^{q_0}L^{r_0}}\\
&\le C(T+ T^\frac1{q_1}(2\rho^\ell)) \,d_X(\p-\ps).
\end{align*}
Thus for smaller $T$ and $\rho$ the mapping $\mathcal N$ is a
contraction and there is a fixed point $\p$ of $\mathcal N$ satisfying \eqref{int eqn}.
The uniqueness and time continuity follows easily from the equation
\eqref{int eqn} and Strichartz estimate. We omit the details.
\end{proof}
\begin{remark}
The mass and energy conservations are straightforward from the uniqueness. One can also show the conservation laws by the argument for Strichartz solutions of \cite{oz}. The global well-posedness follows easily from the conservations in case that $f(x, |\tau|) \le 0$ and $\ell < \frac{4s}{n-2s}$, or $\ell < \frac{4s}{n}$, or $f(x, |\tau|) \le 0$ and $\ell = \frac{4s}{n-2s}$ and $\|\varphi\|_{H^s}$ is small.
\end{remark}


\end{document}